\documentclass[a4, 12pt]{amsart}

\usepackage{amsmath,amssymb,amsthm,amscd}
\usepackage{graphics, graphpap}

\newtheorem{theorem}{Theorem}[section]
\newtheorem{corollary}[theorem]{Corollary}
\newtheorem{lemma}[theorem]{Lemma}

\theoremstyle{definition}

\def\n{\underline n}
\def\Hom{\operatorname{Hom}}
\def\p{\frak p}
\def\m{\frak m}
\def\q{\frak q}

\def\n{\mathbb{N}}

\def\Spec{\operatorname{Spec}}

\def\lim{\operatorname{\lim}}
\def\dlim{{\varinjlim}_n}
 \def\Hom{\operatorname{Hom}}
\def\dim{\operatorname{dim}}
\def\Ext{\operatorname{Ext}}

\def\Im{\operatorname{Im}}
\def\N-dim{\operatorname{N-dim}}

\def\depth{\operatorname{depth}}
\def\f-depth{\operatorname{f-depth}}

\def\Supp{\operatorname{Supp}}

\def\Att{\operatorname{Att}}

\def\Ass{\operatorname{Ass}}

\def\ann{\operatorname{ann}}

\def\Ker{\operatorname{Ker}}

\def\mt{\longrightarrow}
\def\sr{\rightarrow}

\def\Max{\operatorname{Max}}

\begin{document}
\title{\bf On the cofiniteness of generalized local cohomology modules}
\author{Nguyen Tu Cuong$^a$, Shiro Goto$^b$ and Nguyen Van Hoang$^c$}
\date{\small }
\maketitle

\vskip .5cm {\begin{quote} {\bf Abstract}{\footnote{{\it{Key words and phrases: }} Generalized local cohomology, $I$-cofiniteness. \hfill\break
{\it{2000 Subject  Classification:}} 13C15, 13D45  \hfill\break
{$^A$ Institute of Mathematics, 18 Hoang Quoc Viet, Hanoi, Vietnam, \\ {\it E-mail} : ntcuong@math.ac.vn}; 
\hfill\break {$^B$ Department of Mathematics, School of Science and Technology, Meiji University, 1-1-1 Higashi-mita, Tama-ku, Kawasaki 214-8571, Japan, \\{\it E-mail} : goto@math.meiji.ac.jp};
\hfill\break {$^C$ Meiji Institute for Advanced Study of Mathematical Sciences (MIMS), Meiji University
  1-1-1 Higashi-mita, Tama-ku, Kawasaki 214-8571, Japan, \\{\it E-mail} : nguyenvanhoang1976@yahoo.com}}. {\footnotesize Let $R$ be a commutative Noetherian ring, $I$ an ideal of $R$ and $M$, $N$ two finitely generated $R$-modules. The aim of this paper is to investigate the $I$-cofiniteness of  generalized local cohomology modules $\displaystyle H^j_I(M,N)=\dlim\Ext^j_R(M/I^nM,N)$ of $M$ and $N$ with respect to
$I$. 
We first prove that if $I$ is a principal ideal then $H^j_I(M,N)$ is $I$-cofinite for all $M, N$ and all $j$. Secondly, let $t$ be a non-negative integer such that $\dim\Supp(H^j_I(M,N))\le 1\text{ for all }j<t.$ Then $H^j_I(M,N)$ is $I$-cofinite for all $j<t$ and $\Hom(R/I,H^t_I(M,N))$ is finitely generated. Finally, we show that if $\dim(M)\le 2$ or $\dim(N)\le 2$ then $H^j_I(M,N)$ is $I$-cofinite for all $j$.}}
\end{quote}

\section{Introduction} 

Throughout this note the ring $R$ is commutative Noetherian. Let $N$ be finitely generated $R$-modules and $I$ an ideal of $R$. In \cite{Gro}, A. Grothendieck conjectured that if $I$ is an ideal of $R$ and $N$ is a finitely generated $R$-module, then $\Hom_R(R/I,H^j_I(N))$ is finitely generated for all $j\ge 0 $. R. Hartshorne provides a counter-example to this conjecture in \cite{Har}. He also defined an $R-$module $K$ to be $I$-cofinite if $\Supp_R(K)\subseteq V(I)$ and $\Ext^j_R(R/I,K)$ is finitely generated for all $j\ge 0$ and he asked the following question. 
\medskip

\noindent{\bf Question.} {\it For which rings $R$ and ideals $I$ are the modules $H^j_I(N)$ is $I$-cofinite for all $j$ and all finitely generated modules $N$?}
\medskip

\noindent Hartshorne showed that if $N$ is a finitely generated $R$-module, where $R$ is a complete regular local ring, then $H^j_I(N)$ is $I$-cofinite in two cases: \\
$\text{ }$ (i) $I$ is a principal ideal (see \cite[Corollary 6.3]{Har});\\
$\text{ }$ (ii) $I$ is a prime ideal with $\dim(R/I) = 1$ (see \cite[Corollary 7.7]{Har}).\\
K. I. Kawasaki has proved that if $I$ is a principal ideal in a commutative Noetherian ring then $H^j_I(N)$ are $I$-cofinite for all finitely generated $R$-modules $N$ and all $j\ge 0$ (see \cite[Theorem 1]{Kaw1}). D. Delfino and T. Marley \cite[Theorem 1]{DM} and K. I. Yoshida \cite[Theorem 1.1]{Yos} refined result (ii) to more general situation that if $N$ is a finitely generated module over a commutative Noetherian local ring $R$ and $I$ is an ideal of $R$ such that $\dim(R/I) = 1$, then $H^j_I(N)$ are $I$-cofinite for all $j\ge 0$. Recently, K. Bahmanpour and R. Naghipour have extended this result to the case of non-local ring; more precisely, they showed that if $t$ is a non-negative integer such that $\dim\Supp(H^j_I(N))\le 1$ for all $j<t$ then $H^0_I(N), H^1_I(N),\ldots, H^{t-1}_I(N)$ are $I$-cofinite and $\Hom(R/I, H^t_I(N))$ is finitely generated (see \cite[Theorem 2.6]{BaNa}).  

\noindent There are some generalizations of the theory of local cohomology modules. The following generalization of local cohomology theory is given by J. Herzog in \cite{Her}: Let $j$ be a non-negative integer and $M$ a finitely generated $R$-module. Then the $j^{th}$ generalized local cohomology module of $M$ and $N$ with respect to $I$ is defined by
 $$\displaystyle H^j_I(M,N)=\dlim\Ext^j_R(M/I^nM,N).$$ 
These modules were studied further in many research papers such as: \cite{Su}, \cite{Ya1}, \cite{Bz}, \cite{HZ}, \cite{Ya2}, \cite{Kaw2}, \cite{Hoa}, \cite{CH1}, \cite{Char1}, \cite{Char2}, \ldots. It is clear that $H^j_I(R,N)$ is just the ordinary local cohomology module $H^j_I(N)$. 
\medskip

\noindent The purpose of this paper is to investigate a similar question as above for the theory of generalized local cohomology.  Our first main result is the following theorem.
\begin{theorem}\label{dlc} If $I$ is a principal ideal then $H^j_I(M,N)$ is $I$-cofinite for all finitely generated $R-$modules $M, N$ and all $j$. 
\end{theorem}
 \noindent As an immediate consequence of this theorem, we obtain again a theorem of K. I. Kawasaki \cite[Theorem 1]{Kaw1} (see Corollary \ref{mrkaw1}). Moreover, Theorem \ref{dlc} is an improvement of \cite[Theorem 2.8]{DS}, since we do not need the hypothesis that $M$ has finite projective dimension as in \cite{DS}.  It should be noticed that the arguments of local cohomology that used in the proof  of K. I. Kawasaki \cite{Kaw1} can not apply  to proving Theorem \ref{dlc}. Because, for the case of local cohomology, if $I$ is a principal ideal then $H^j_I(N)=0$ for all $j>1$. But this does not happen in the theory of generalized local cohomology, i.e. $H^j_I(M,N)$ may not vanish for $j>1$ even if $I$ is principal ideal. Therefore,  we  have to use a criterion on the cofiniteness which was invented by L. Melkersson in \cite{Me1}.  Here we also give a more elementary proof for this criterion (see Lemma \ref{tccofi}). The next theorem is our second main result in this paper.

\begin{theorem}\label{Thm1} Let $t$ be a non-negative integer such that $\dim\Supp(H^j_I(M,N))\le 1$ for all $j<t$. Then 
$H^j_I(M,N)$ is $I$-cofinite for all $j<t$ and $\Hom(R/I,H^t_I(M,N))$ is finitely generated.
\end{theorem}
\medskip

\noindent This theorem is an extension for  generalized local cohomology modules of a result of K. Bahmanpour and R. Naghipour
 \cite[Theorem 2.6]{BaNa}. In \cite{BaNa}, they had used a basic property of local cohomology that $H^j_I(N)\cong H^j_I(N/\Gamma_I(N))$ for all $j>0$; then it is easy to reduce to the case of $\Gamma_I(N)=0$. But, it is not true that $H^j_I(M,N)\cong H^j_I(M,N/\Gamma_{I_M}(N))$ for all $j>0$ in general, where $I_M=\ann_R(M/IM)$. Hence, we need to establish Lemma \ref{kte} which says that if $t$ and $k$ are non-negative integers such that $\dim\Supp(H^j_I(M,N))\le k$ for all $j<t$ then so is $H^j_I(M,N/\Gamma_{I_M}(N))$. Moreover, in order to prove Theorem \ref{Thm1}, we also need some more auxiliary lemmas such as \ref{Lm3}, \ref{Min}, \ref{ATn} on minimax modules. Especially, by Lemma \ref{BdMJ}, instead of studying the cofiniteness of $H^j_I(M,N)$, we need only to prove the cofiniteness of these modules with respect to $I_M$.  As a consequence of Theorem \ref{Thm1}, we prove that if $\dim\Supp(H^j_I(M,N))\le 1$ for all $j$ (this is the case, for example  if $\dim(N/I_MN)\le 1$) then $H^j_I(M,N)$ is $I$-cofinite for all $j$ (Corollary \ref{CQ1}). This is an improvement of \cite[Theorem 2.9]{DS}, because our theorem does not need the hypothesis that $R$ is complete local, $M$ is of finite projective dimension, and $I$ is prime ideal with $\dim(R/I)=1$. An other consequence of Theorem \ref{Thm1} on the finiteness of Bass numbers is Corollary \ref{sbn} which is a stronger result than the main result of S. Kawakami and K. I. Kawasaki in \cite{Kaw2}.
\medskip

\noindent On the other hand, in the case of small dimension, the third author in \cite[Lemma 3.1]{Hoa} proved that if $\dim(N)\le 2$ then any quotient of $H^j_I(M,N)$ has only finitely many associated prime ideals for all finitely generated $R-$modules $M$ and all $j\ge 0$. We can now prove  a stronger result in the following theorem. 

\begin{theorem}\label{Thm3} 
Assume that $\dim(M)\le 2$ or $\dim(N)\le 2$. Then $H^j_I(M,N)$ is $I$-cofinite for all $j$.
\end{theorem}

\noindent As an immediate consequence of Theorem \ref{Thm3}, we get a result on the cofiniteness of local cohomology modules (see Corollary \ref{Dchl}). Moreover, by application of Theorems \ref{Thm1} and \ref{Thm3}, we obtain a finiteness result on the set of associated prime ideals of $\Ext^i_R(R/I,H^j_I(M,N))$ for all $i, j\ge 0$ when $(R,\m)$ is a Noetherian local ring and $\dim(M)\le 3$ or $\dim(N)\le 3$ (Corollary \ref{Pro1}). 
\medskip

The paper is divided into five sections. In Section 2, we prove some auxiliary lemmas which will be used in the sequel. Section 3, 4 and 5 are devoted to prove three main results and its consequences.

\section{Auxiliary lemmas}

Let $R$ be a commutative Noetherian ring, $I$ an ideal of $R$, and $M$, $N$ finitely generated $R$-modules. We always denote by $I_M$ the annihilator of $R$-module $M/IM$, i.e. $I_M=\ann_R(M/IM)$. We first recall the following lemma.

\begin{lemma}\label{Lmbs} {\rm (cf. \cite[Lemma 2.3]{CH} and \cite[Lemma 2.1]{CH1})} 
\item[\ (i)] If $I\subseteq\ann(M)$ or $\Gamma_I(N)=N$ then $H^j_I(M,N)\cong\Ext^j_R(M,N)$ for all $j\ge 0$.
\item[\ (ii)] $H^j_I(M,N)$ is $I_M$-torsion.
\end{lemma}

We next prove some auxiliary lemmas which will be used in sequel.

\begin{lemma}\label{kte} Let $t$ and $k$ be non-negative integers. If $\dim\Supp(H^j_I(M,N))\le k$ for all $j<t$, then so is $H^j_I(M,N/\Gamma_{I_M}(N))$.
\end{lemma}
\begin{proof} From the short exact sequence $0\sr\Gamma_{I_M}(N)\sr N\sr N/\Gamma_{I_M}(N)\sr 0,$ 
we get the long exact sequence
{\small $$...\sr\Ext^j_R(M,\Gamma_{I_M}(N))\sr H^j_I(M,N)\sr H^j_I(M,\overline N)\sr\Ext^{j+1}_R(M,\Gamma_{I_M}(N))\sr...,$$}
for all $j$, where $\overline N= N/\Gamma_{I_M}(N)$. 
We assume that there exists an integer $i<t$ and $\p\in\Supp(H^i_I(M,\overline N))$ such that $\dim(R/\p)>k$ and $\p\notin\Supp(H^j_I(M,\overline N))$ for all $j<i$. Thus by the long exact sequence as above we obtain the following exact sequence
{\small $$...\sr\Ext^j_R(M,\Gamma_{I_M}(N))_\p\sr H^j_I(M,N)_\p\sr H^j_I(M,\overline N)_\p\sr\Ext^{j+1}_R(M,\Gamma_{I_M}(N))_\p\sr...$$}
Note that $H^j_I(M,N)_\p=0$ for all $j\le i$, while $H^j_I(M,\overline N)_\p=0$ for all $j<i$, and $H^i_I(M,\overline N)_\p\ne 0$. So, by the above exact sequence, we have $\Ext^j_R(M,\Gamma_{I_M}(N))_\p=0$ for all $j\le i$, and $\Ext^{i+1}_R(M,\Gamma_{I_M}(N))_\p\ne 0$. It implies that $\Gamma_{I_M}(N)_\p\ne 0$ and  $$\depth(\ann(M)_\p,\Gamma_{I_M}(N)_\p)=i+1\ge 1.$$ Hence $\ann(M)_\p\nsubseteq\q R_\p$ for all $\q R_\p\in\Ass_{R_\p}(\Gamma_{I_M}(N)_\p)$. This contradicts with the fact that $\Ass_{R_\p}(\Gamma_{I_M}(N)_\p)=\Ass_{R_\p}(N_\p)\cap V((I_M)_\p)$ and $I_M\supseteq \ann(M)$. 
\end{proof}

In \cite{Zo}, H. Z${\rm\ddot o}$schinger had introduced the class of minimax modules. An $R$-module $K$ is said to be a
{\it minimax module}, if there is a finitely generated submodule $T$ of $K$, such that $K/T$ is Artinian. Thus the class
of minimax modules includes all finitely generated and all Artinian modules.

\begin{lemma}\label{Lm3}  Let $t$ be a non-negative integer such that $H^j_I(M,N)$ is $I$-cofinite minimax for all $j<t$. Then $\Hom_R(R/I,H^t_I(M,N))$ is finitely generated. In particular, $\Ass(H^t_I(M,N))$ is a finite set.
\end{lemma}

\begin{proof} We prove by induction on $t\ge 0$. If $t=0$ then the result is trivial. Assume that $t>0$ and the result holds true for $t-1$. From the short exact sequence $0\sr\Gamma_I(N)\sr N\sr \overline N\sr 0$, we get the long exact sequence 
$$\Ext^j_R(M,\Gamma_I(N))\xrightarrow{f_j} H^j_I(M,N)\xrightarrow{g_j} H^j_I(M,\overline N)\xrightarrow{h_j}\Ext^{j+1}_R(M,\Gamma_I(N)),$$
where $\overline N=N/\Gamma_I(N)$. For each $j\ge 0$ we split the above exact sequence into two the following exact sequences
$$0\sr\Im f_j\sr H^j_I(M,N)\sr\Im g_j\sr 0\quad\text{ and }$$
$$0\sr\Im g_j\sr H^j_I(M,\overline N)\sr\Im h_j\sr 0.$$
Note that $\Im f_j$ and $\Im h_j$ is finitely generated for all $j\ge 0$. Then, $H^j_I(M,N)$ is $I$-cofinite if and only if $\Im g_j$ is $I$-cofinite if and only if $H^j_I(M,\overline N)$ is $I$-cofinite for all $j\ge 0$; and we get by \cite[Lemma 2.1]{Bak} that $H^j_I(M,N)$ is minimax if and only if so is $H^j_I(M,\overline N)$ for all $j\ge 0$. Hence $H^j_I(M,N)$ is $I$-cofinite minimax if and only if so is $H^j_I(M,\overline N)$ for all $j\ge 0$. Moreover, from the two exact sequences above, we get the following exact sequences
$$
\begin{aligned}
\Hom(R/I,\Im f_j)\sr\Hom(R/I,&H^j_I(M,N))\\
\sr&\Hom(R/I,\Im g_j)\sr\Ext^1_R(R/I,\Im f_j)\text{ \quad and }
\end{aligned}
$$
$$
\begin{aligned}
0\sr\Hom(R/I,\Im g_j)\sr\Hom(R/I,H^j_I(M,\overline N))\sr\Hom(R/I,\Im h_j)
\end{aligned}
$$
for all $j\ge 0$. Thus we obtain that $\Hom(R/I,H^j_I(M,N))$ is finitely generated if and only if so is $\Hom(R/I,H^j_I(M,\overline N))$ for all $j\ge 0$. Therefore, in order to prove the lemma, we may assume that $\Gamma_I(N)=0$. Thus, there exists $x\in I$ such that $0\sr N\xrightarrow{x} N\sr N/xN\sr 0$ is exact. From this, we obtain the short exact sequence 
$$0\sr H^{t-1}_I(M,N)/xH^{t-1}_I(M,N)\sr H^{t-1}_I(M,N/xN)\sr (0:x)_{H^t_I(M,N)}\sr 0$$
Hence we get the following exact sequence
$$
\begin{aligned}
\Hom(R/I,H^{t-1}_I(M,N/xN))\sr&\Hom(R/I,(0:x)_{H^t_I(M,N)})\\
\sr&\Ext^1_R(R/I,H^{t-1}_I(M,N)/xH^{t-1}_I(M,N)).
\end{aligned}
$$
By the inductive hypothesis, $\Hom(R/I,H^{t-1}_I(M,N/xN))$ is finitely generated. Moreover, by the assumption, $H^{t-1}_I(M,N)$ is $I$-cofinite minimax. It implies by \cite[Corollary 4.4]{Me1} that $\Ext^1_R(R/I,H^{t-1}_I(M,N)/xH^{t-1}_I(M,N))$ is finitely generated. Therefore, $\Hom(R/I,H^t_I(M,N))\cong\Hom(R/I,(0:x)_{H^t_I(M,N)})$ is finitely generated as required.
\end{proof}

From Lemma \ref{Lm3}, we get a short proof for the following result which was proved first by Yassemi-Khatami-Sharif in \cite[Theorem 2.1]{Ya2}.

\begin{corollary}\label{YaKh} If $H^0_I(M,N), H^1_I(M,N),\ldots, H^{t-1}_I(M,N)$ is finitely generated then $\Ass(H^t_I(M,N)/T)$ is a finite set for any finitely generated submodule $T$ of $H^t_I(M,N)$.
\end{corollary}
\begin{proof} By the short exact sequence $0\sr T\sr H^t_I(M,N)\sr H^t_I(M,N)/T\sr 0,$ we get the following exact sequence
$$\Hom(R/I,H^t_I(M,N))\sr \Hom(R/I,H^t_I(M,N)/T)\sr\Ext^1_R(R/I,T).$$
Thus, since $\Hom(R/I,H^t_I(M,N))$ and $\Ext^1_R(R/I,T)$ are finitely generated by Lemma \ref{Lm3} and by the assumption of $T$, we obtain that $\Hom(R/I,H^t_I(M,N)/T)$ is also finitely generated. Thus the result follows.
\end{proof}

\begin{lemma}\label{Min}  Let $t$ be a non-negative
integer such that $H^j_I(M,N)$ is minimax for all $j<t$. Then $\Hom_R(R/I,H^t_I(M,N))$ is finitely generated
and $H^j_I(M,N)$ are $I$-cofinite for all $j<t$.
\end{lemma}
\begin{proof} By Lemma \ref{Lm3}, we need only to prove that $H^j_I(M,N)$ is $I$-cofinite for all $j<t$. We proceed by induction on $j$. It is clear that $H^0_I(M,N)$ is $I$-cofinite. Assume that $j>0$ and the result holds true for smaller values than $j$. Thus we obtain that $H^0_I(M,N), \ldots, H^{j-1}_I(M,N)$ are $I$-cofinite minimax by the inductive hypothesis and by the hypothesis. It follows by Lemma \ref{Lm3} that $\Hom(R/I,H^j_I(M,N))$ is finitely generated. So that $H^j_I(M,N)$ is $I$-cofinite by \cite[Proposition 4.3]{Me1} as required.
\end{proof}

\begin{lemma}\label{ATn} 
Let $t$ be a non-negative integer such that $\Supp(H^j_I(M,N))\subseteq\Max(R)$ for all $j<t$. Then $H^j_I(M,N)$ is Artinian for all $j<t$.
\end{lemma}
\begin{proof}
We now prove the lemma by induction on $t$. If $t=1$ then it is clear that $H^0_I(M,N)$ is Artinian. Assume that $t\ge 2$ and the lemma holds true for $t-1$. By the inductive hypothesis, the $R$-modules $H^j_I(M,N)$ is Artinian for all $j<t-1$. Therefore, by Lemma \ref{Min},  $\Hom(R/I,H^{t-1}_I(M,N))$ is finitely generated. Thus, since $\Supp(\Hom(R/I,H^{t-1}_I(M,N)))\subseteq\Max(R)$, we obtain that $\Hom(R/I,H^{t-1}_I(M,N))$ is Artinian. On the other hand, as $H^{t-1}_I(M,N)$ is $I$-torsion, it follows by \cite[Theorem 1.3]{Me2} that $H^{t-1}_I(M,N)$ is Artinian.
\end{proof}

\section{Proof of Theorem \ref{dlc}}

We first need the following lemma which has been proved in \cite[Corollary 3.4]{Me1} by L. Melkersson. We  give here an another proof  for this result with elementary arguments.

\begin{lemma}\label{tccofi}  Let $K$ be an $R$-module. Suppose $x\in I$ and $\Supp(K)\subset V(I)$. If $(0:x)_K$ and $K/xK$ are both $I$-cofinite, then $K$ must be $I$-cofinite.
\end{lemma}
\begin{proof} Let $t$ be a non-negative integer. We need only to claim that $\Ext^t_R(R/I,K)$ is finitely generated. By the commutative diagram
\[
\begin{aligned}
0\sr (0:_Kx)\sr&K\xrightarrow{\text{ }\ x} xK\mt 0\\
x&\downarrow\quad\searrow x\\
0\mt&xK\mt K\sr K/xK\sr 0 
\end{aligned}
\]
we obtain the following commutative diagram of long exact sequences
\[
\begin{aligned}
...\sr\Ext^t_R(R/I,(0:_Kx))\sr&\Ext^t_R(R/I,K)\xrightarrow{x^{(t)}}\Ext^t_R(R/I,xK)\sr...\\
&x^{(t)}\downarrow \qquad\qquad\quad\searrow x \\
...\sr\Ext^{t-1}_R(R/I,K/xK)\sr&\Ext^t_R(R/I,xK)\xrightarrow{f_t}\Ext^t_R(R/I,K)\sr...
\end{aligned}
\]
where $x^{(t)}=\Ext^t_R(R/I,x)$. Note that $K/xK$ is $I$-cofinite by the hypothesis, it implies that $\Ext^{t-1}_R(R/I,K/xK)$ is finitely generated. Thus $\Ker(f_t)$ is finitely generated. Moreover, since the triangle is commutative, so that $x^{(t)}\big((0:x)_{\Ext^t_R(R/I,K)}\big)\subseteq\Ker(f_t).$ 
It follows that $x^{(t)}\big((0:x)_{\Ext^t_R(R/I,K)}\big)$ is finitely generated.
On the other hand, $(0:x)_K$ is $I$-cofinite by the hypothesis, so we obtain that $\Ext^t_R(R/I,(0:x)_K)$ is finitely generated. It implies that $\Ker(x^{(t)})$ is finitely generated.
Therefore, by the following exact sequence
\[
\begin{aligned}
0\sr\Ker(x^{(t)})\cap (0:x)_{\Ext^t_R(R/I,K)}\sr (0:x)&_{\Ext^t_R(R/I,K)}\\
&\mt x^{(t)}\big((0:x)_{\Ext^t_R(R/I,K)}\big)\sr 0,
\end{aligned}
\]
we obtain that $(0:x)_{\Ext^t_R(R/I,K)}$ is finitely generated. Finally, note that $x\in I$, it yields that 
$\Ext^t_R(R/I,K)=(0:x)_{\Ext^t_R(R/I,K)}$ 
is finitely generated as required.
\end{proof}

We now are ready to prove Theorem \ref{dlc}.

\begin{proof}[\bf Proof of Theorem \ref{dlc}]
Assume that $I=Rx$ is a principal ideal. From the short exact sequence 
$$0\sr\Gamma_I(M)\sr M\sr \overline M\sr 0,$$ 
where $\overline M=M/\Gamma_I(M)$, we get by \cite{HZ} the following exact sequence 
$$H^{i-1}_I(\Gamma_I(M),N)\sr H^i_I(\overline M,N)\sr H^i_I(M,N)\sr H^i_I(\Gamma_I(M),N)$$
for all $i$. Since $\Gamma_I(M)=(0:I^k)_M$ for some positive integer $k$, we get by Lemma \ref{Lmbs} that 
$$H^i_I(\Gamma_I(M),N)=H^i_{I^k}((0:I^k)_M,N)\cong\Ext^i_R(\Gamma_I(M),N)$$
for all $i$. Hence $H^i_I(\Gamma_I(M),N)$ is finitely generated for all $i$, it follows by the above exact sequence that $H^i_I(M,N)$ is $I$-cofinite if and only if so is $H^i_I(\overline M,N)$. Hence we may assume that $\Gamma_I(M)=0$. So that $I\nsubseteq\p$ for all $\p\in\Ass(M)$. It implies that $x\notin\p$ for all $\p\in\Ass(M)$. Thus we obtain an exact sequence 
$$0\sr M\xrightarrow{x} M\sr M/xM\sr 0.$$ 
From this we have the following exact sequence
$$0\sr H^{i-1}_I(M,N)/xH^{i-1}_I(M,N)\sr H^i_I(M/xM,N)\sr (0:x)_{H^i_I(M,N)}\sr 0$$
for all $i$. Note that, as $I=Rx$, so we obtain by Lemma \ref{Lmbs} that
$$H^i_I(M/xM,N)\cong\Ext^i_R(M/xM,N)$$ 
for all $i$. Hence $H^i_I(M/xM,N)$ is finitely generated for all $i$. Thus by the above exact sequence we obtain that 
$$(0:x)_{H^i_I(M,N)}\text{ and }H^i_I(M,N)/xH^i_I(M,N)$$ 
are finitely generated for all $i$. Therefore we get by Lemma \ref{tccofi} that $H^i_I(M,N)$ is $I$-cofinite for all $i$.
\end{proof}

By replacing $M$ by $R$ in Theorem \ref{dlc} we obtain a theorem of K. I. Kawasaki on the cofiniteness of local cohomology modules as follows.

\begin{corollary}\label{mrkaw1} {\rm(\cite[Theorem 1]{Kaw1})} If $I$ is a principal ideal, then $H^j_I(N)$ is $I$-cofinite for all finitely generated $R-$module $N$ and all $j$.
\end{corollary}

\section{Proof of Theorem \ref{Thm1}}
Before proving Theorem \ref{Thm1}, we need to recall some known facts on the theory of secondary representation. 

In \cite{Mac}, I. G. Macdonald has developed the theory of attached prime ideals and secondary representation of a module, which is (in a certain sense) a dual to the theory of associated prime ideals and primary decompositions. A non-zero $R$-module $K$ is called secondary if for each $a\in R$ multiplication by $a$ on $K$ is either surjective or nilpotent. Then $\p = \sqrt{\ann(K)}$ is a prime ideal and $K$ is called $\p$-secondary. We say that $K$ has a secondary representation if there is a finite number of secondary submodules $K_1, K_2,\ldots, K_n$ such that $K = K_1 +K_2 +\ldots+K_n.$ 
One may assume that the prime ideals $\p_i = \sqrt{\ann(K_i)}, i = 1, 2,\ldots, n$ are all distinct, and by omitting redundant summands, that the representation is minimal. Then the set of prime ideals $\{\p_1,\ldots,\p_n\}$ does not depend on the representation, and it is called the set of attached prime ideals of $K$ and denoted by $\Att(K).$ Note that if $A$ is an Artinian $R$-module then $A$ has a secondary representation. The basic properties on the set $\Att(A)$ of attached primes of $A$ are referred in a paper by I. G. Macdonal \cite{Mac}. If $0 \sr A_1\sr A_2\sr A_3\sr 0$ is an exact sequence of Artinian $R$-modules then $$\Att(A_3)\subseteq \Att(A_2)\subseteq\Att(A_1)\cup\Att(A_3).$$

\begin{lemma}\label{TCA} Let $x$ be an element of $R$,  $I$ an ideal of $R$ and $A$ an Artinian $R$-module. Then the following statements are true.
\begin{itemize}
\item[(i)] If $x\notin\p$ for all $\p\in\Att(A)\setminus\Max(R)$, then $\ell(A/xA)<\infty$.
\item[(ii)] If $(0:I)_A$ is finitely generated, then $I\nsubseteq\p$ for all $\p\in\Att(A)\setminus\Max(R)$. 
\end{itemize}
\end{lemma}
\begin{proof} $(i)$ Assume that $\Att(A)\setminus\Max(R)=\{\p_1,\ldots,\p_n\}$. Let 
$$A=A_1+\ldots+A_n+B_1+\ldots+B_t$$ 
be a minimal secondary representation of $A$, where $A_i$ is $\p_i$-secondary and $B_j$ is $\m_i$-secondary for all $i=1,\ldots,n$ and all $j=1,\ldots,t$ (with $\m_j\in\Max(R)$ for all $j=1,\ldots,t$). Set $B=B_1+\ldots+B_t$, then $\Att(B)\subseteq\Max(R)$. Since $x\notin\p_i$ for all $i=1,\ldots,n$, so that $xA_i=A_i$ for all $i=1,\ldots,n$. It follows that $xA=A_1+\ldots+A_n+xB$. Note that 
\begin{alignat}{2}
A/xA&=\big((A_1+\ldots+A_n+xB)+B\big)/\big(A_1+\ldots+A_n+xB\big)\notag\\
&\cong B/\big(B\cap (A_1+\ldots+A_n+xB)\big).\notag
\end{alignat}
Therefore  
$\Att(A/xA)\subseteq\Att\big(B/(B\cap (A_1+\ldots+A_n+xB))\big)\subseteq\Att(B)\subseteq\Max(R).$
From this, since $A/xA$ is Artinian, so that $\ell(A/xA)<\infty$.

\noindent (ii) We first claim that $\sqrt{\ann(0:_AI)}=\sqrt{\ann(0:_AI^n)}$ for all $n\ge 2$. Consider $n=2$, it is clear that $\sqrt{\ann(0:_AI)}\supseteq\sqrt{\ann(0:_AI^2)}$. Conversely, for any $a\in\sqrt{\ann(0:_AI)}$ then there is an integer $t>0$ such that $a^t(0:_AI)=0$. We now prove that $a^{2t}(0:_AI^2)=0$ (and therefore $a\in\sqrt{\ann(0:_AI^2)}$\  ). Indeed, for any $y\in (0:_AI^2)$, then $I^2y=0$. So that $Iy\subseteq(0:_AI)$, thus $a^t(Iy)=0$. Hence $a^ty\in (0:_AI)$, and thus $a^t(a^ty)=0$. Therefore $a^{2t}y=0$. We now assume that $n>2$ and the claim is true for $n-1$. Let $a\in\sqrt{\ann(0:_AI)}$ then by induction assumption $a\in\sqrt{(0:_AI^{n-1})}$. Thus $a^t(0:_AI^{n-1})=0$ for some $t>0$. For any $y\in(0:_AI^n)$, then $I^{n-1}Iy=I^ny=0$. Hence $Iy\subseteq(0:_AI^{n-1})$, so that $I(a^ty)=a^t(Iy)=0$. It implies that $a^ty\in(0:_AI)$. On the other hand, since $a\in\sqrt{\ann(0:_AI)}$, so $a^l(0:_AI)=0$ for some $l>0$. Therefore $a^{t+l}y=0$, it yields that $a\in\sqrt{(0:_AI^n)}$. So we get the claim. Finally for any $\p\in\Att(A)\setminus\Max(R)$ we obtain that $I\nsubseteq\p$. Indeed, assume that $I\subseteq\p$ for some $\p\in\Att(A)\setminus\Max(R)$. Then there exists a submodule $U$ of $A$ such that $U$ is $\p$-secondary. Thus there is an integer $n$ such that $\p^nU=0$. Hence, as $I\subseteq\p$, so that $I^nU=0$. Therefore $U=(0:_UI^n)\subseteq (0:_AI^n)$. Hence since $\ell(0:_AI)<\infty$ then we get by the claim that $(0:_AI^n)$ is of finite length. It implies that $\ell(U)<\infty$, so $\p\in\Max(R)$, this is a contradiction. 
\end{proof}

\begin{lemma}\label{BdMJ} Let $t$ be a non-negative integer. Then 
\begin{itemize}
\item[(i)] $H^t_I(M,N)$ is $I-$cofinite if and only if $H^t_I(M,N)$ is $I_M-$cofinite, where $I_M=\ann_R(M/IM)$.
\item[(ii)] $\Hom(R/I,H^t_I(M,N))$ is finitely generated if and only if so is $\Hom(R/I_M,H^t_I(M,N))$.
\end{itemize}
\end{lemma}
\begin{proof} Set $K=H^t_I(M,N)$. Note that $\Supp(K)\subseteq \Supp(R/I_M)\subseteq \Supp(R/I)$.
\item (i) If $K$ is $I-$cofinite then, since $I\subseteq I_M$, we get that $K$ is $I_M-$cofinite by \cite[Proposition 1]{DM}. Assume that $K$ is $I_M-$cofinite. Thus, as  $\sqrt{I_M}=\sqrt{I+\ann(M)}$, $K$ is $(I+\ann(M))-$cofinite by \cite[Proposition 1]{DM}. Let $x_1,\ldots,x_t, y_1,\ldots,y_s$ be generators of $I+\ann(M)$ such that $I=(x_1,\ldots,x_t)$ and $\ann(M)=(y_1,\ldots,y_s)$. Then Koszul cohomology modules $H^j(\underline x, y_1,\ldots,y_s; K)$ are finitely generated $R-$modules for all $j$ by \cite[Theorem 1.1]{Me0} (here we set $\underline x=x_1,\ldots,x_t$ for short). We now claim by descending induction on $l$ (with $0\le l\le s$) that 
$H^j(\underline x, y_1,\ldots,y_l; K)$ are finitely generated $R-$modules for all $j$, where we use the convention that $H^j(\underline x; K)=H^j(\underline x, y_1,\ldots,y_l; K)$ if $l=0$. If $l=s$ then the claim is clear. Suppose $l<s$ and $H^j(\underline x, y_1,\ldots,y_{l+1}; K)$ are finitely generated $R-$modules for all $j$. We first consider the case $j=0$. As $y_{l+1}\in\ann(K)$, so we get that
\begin{alignat}{2}
H^0(\underline x,y_1,\ldots,y_l;K)&\cong (0:_K(\underline x,y_1,\ldots,y_l)R)\notag\\
&\cong (0:_K(\underline x,y_1,\ldots,y_l,y_{l+1})R)\notag\\
&\cong H^0(\underline x,y_1,\ldots,y_l, y_{l+1};K).\notag
\end{alignat}
Thus $H^0(\underline x,y_1,\ldots,y_l;K)$ is a finitely generated $R-$module. Assume that $j\ge 1$. We consider the following exact sequences (cf. \cite[Section 5]{Me3})
$$H^{j-1}(\underline x, y_1,\ldots,y_l,y_{l+1};K)\rightarrow H^{j}(\underline x, y_1,\ldots,y_l;K)\xrightarrow{y_{l+1}} H^{j}(\underline x, y_1,\ldots,y_l;K)$$
for all $j\ge 1$. Since $y_{l+1}\in\ann(K)$, so that $y_{l+1}H^{j}(\underline x, y_1,\ldots,y_l;K)=0$. Hence, the above exact sequence implies that the following sequence
$$H^{j-1}(\underline x, y_1,\ldots,y_l,y_{l+1};K)\rightarrow H^{j}(\underline x, y_1,\ldots,y_l;K)\rightarrow0$$
is exact for all $j\ge 1$. From this we get by induction assumption that $H^j(\underline x, y_1,\ldots,y_l;K)$ are finitely generated $R-$modules for all $j\ge 1$. Thus the claim is proved. In particular,  $H^j(\underline x; K)$ are finitely generated $R-$modules for all $j$. Therefore, we get by \cite[Theorem 1.1]{Me0} again that $K$ is $I-$cofinite.

\item (ii) We note that $\Hom(R/I+\ann(M),K)\cong\Hom(R/I, K)$, as $\ann(M)\subseteq\ann(K)$. Hence, since $\sqrt{I+\ann(M)}=\sqrt{I_M}$, the result follows by \cite[Proposition 1]{DM}.
\end{proof}

We are now ready to prove Theorem \ref{Thm1}.

\begin{proof}[Proof of Theorem \ref{Thm1}]
By Lemma \ref{BdMJ}, we need only to claim that $H^j_I(M,N)$ is $I_M-$cofinite for all $j<t$ and $\Hom(R/I_M, H^t_I(M,N))$ is finitely generated, provided $\dim\Supp(H^j_I(M,N))\le 1$ for all $j<t$ (where t is a given integer).

We prove the claim by induction on $t\ge 0$. The case of $t=0$ is trivial. If $t=1$ then it is clear that $H^0_I(M,N)$ is $I_M$-cofinite; moreover we get by Lemma \ref{Min} that $\Hom(R/I_M, H^1_I(M,N))$ is finitely generated. Assume that $t>1$ and the result holds true for the case $t-1$. From the short exact sequence $0\sr\Gamma_{I_M}(N)\sr N\sr \overline N\sr 0,$ we get the long exact sequence 
$$\Ext^j_R(M,\Gamma_{I_M}(N))\xrightarrow{f_j} H^j_I(M,N)\xrightarrow{g_j} H^j_I(M,\overline N)\xrightarrow{h_j}\Ext^{j+1}_R(M,\Gamma_{I_M}(N)),$$
where $\overline N=N/\Gamma_{I_M}(N)$. For each $j\ge 0$ we split the above exact sequence into two the following exact sequences
$$0\sr\Im f_j\sr H^j_I(M,N)\sr\Im g_j\sr 0\text{ and}$$ 
$$0\sr\Im g_j\sr H^j_I(M,\overline N)\sr\Im h_j\sr 0.$$
Note that $\Im f_j$ and $\Im h_j$ is finitely generated for all $j\ge 0$. Then, for each $j<t$, we obtain that $H^j_I(M,N)$ is $I_M$-cofinite if and only if so is $H^j_I(M,\overline N)$. On the other hand, we get by Lemma \ref{kte} that $\dim\Supp(H^j_I(M,\overline N))\le 1$ for all $j<t$. Therefore, in order to prove the theorem for the case of $t>1$, we may assume that $\Gamma_{I_M}(N)=0$. Hence $I_M\nsubseteq\bigcup_{\p\in\Ass_R(N)}\p.$ 
Set $$X=\bigcup_{j=0}^{t-1}\Supp(H^j_I(M,N))\text{ and } S=\{\p\in X\mid\dim(R/\p)=1\}.$$ 
Thus $S\subseteq\bigcup_{j=0}^{t-1}\Ass(H^j_I(M,N))$. Note that $H^j_I(M,N)$ is $I_M$-cofinite for all $j<t-1$ and $\Hom(R/I_M,H^{t-1}_I(M,N))$ is finitely generated by the inductive hypothesis. It implies that $\bigcup_{j=0}^{t-1}\Ass(H^j_I(M,N))$ is a finite set, and so $S$ is a finite set. Assume that $S=\{\p_1, \p_2,\ldots,\p_n\}$. Then it is clear that
$$\Supp_{R_{\p_k}}(H^j_{IR_{\p_k}}(M_{\p_k},N_{\p_k}))\subseteq \Max(R_{\p_k})$$ 
for all $j<t$ and all $k=1,\ldots,n$. From this, we get by Lemma \ref{ATn} that $H^j_{IR_{\p_k}}(M_{\p_k},N_{\p_k})$ is Artinian for all $j<t$ and all $k=1,\ldots,n$. Note that $V(I_M)\subseteq V(I)$. Hence, it implies by Lemma \ref{Min} and \cite[Proposition 1]{DM} that  
$\Hom(R_{\p_k}/(I_M)R_{\p_k}, H^j_{IR_{\p_k}}(M_{\p_k},N_{\p_k}))$ is finitely generated for all $j<t$ and all $k=1,\ldots,n$. Therefore it yields by Lemma \ref{TCA}(ii) that $$V((I_M)R_{\p_k})\cap\Att_{R_{\p_k}}(H^j_{IR_{\p_k}}(M_{\p_k},N_{\p_k}))\subseteq\Max(R_{\p_k})$$ for all $j<t$ and all $k=1,\ldots,n$. Let 
$$ T= \bigcup_{j=0}^{t-1}\bigcup_{k=1}^n\{\q\in\Spec R\mid\q R_{\p_k}\in\Att_{R_{\p_k}}(H^j_{IR_{\p_k}}(M_{\p_k},N_{\p_k}))\}.$$ 
Then we have $T\cap V(I_M)\subseteq S$. We now choose an element $x\in I_M$ such that 
$$x\notin\Big(\bigcup_{\p\in T\setminus V(I_M)}\p\Big)\cup\Big(\bigcup_{\p\in\Ass_R(N)}\p\Big).$$
Thus, we have the short exact sequence $0\sr N\xrightarrow{x} N\sr N/xN\sr 0.$ It implies the following exact sequence
$$\begin{aligned}
H^j_I(M,N)\xrightarrow{x}H^j_I(M,N)\sr H^j_I(M,N/xN)\sr H^{j+1}_I(M,N)
\end{aligned}$$
for all $j\ge 0$. Thus, we have an exact sequence 
\begin{equation}0\sr H^j_I(M,N)/xH^j_I(M,N)\xrightarrow{\alpha_j} H^j_I(M,N/xN)\xrightarrow{\beta_j} (0:x)_{H^{j+1}_I(M,N)}\sr 0\tag{1}\end{equation} 
for all $j\ge 0$. Note that $\dim\Supp(H^j_I(M,N/xN))\le 1$ for all $j<t-1$ by the above exact sequence and by the hypothesis. So that, we get by the induction assumption that $H^0_I(M,N/xN), H^1_I(M,N/xN),\ldots, H^{t-2}_I(M,N/xN)$ are $I_M$-cofinite and $\Hom(R/I_M, H^{t-1}_I(M,N/xN))$ is finitely generated. Moreover, also by the induction assumption, we have $H^0_I(M,N), H^1_I(M,N),\ldots, H^{t-2}_I(M,N)$ are $I_M$-cofinite and $\Hom(R/I_M, H^{t-1}_I(M,N))$ is finitely generated. For each $j<t$, we set $L_j=H^j_I(M,N)/xH^j_I(M,N)$. By the choice of $x$ and by Lemma \ref{TCA},  we obtain  that $(L_j)_{\p_k}$ has finite length for all $j<t$ and all $k=1,\ldots, n$. From this by the Noetherianness of  $(L_j)_{\p_k}$, there exists a finitely generated submodule $L_{jk}$ of $L_j$ such that $(L_j)_{\p_k}=(L_{jk})_{\p_k}$ for any $j<t$ and any $k=1,\ldots, n$. Let $L'_j=L_{j1}+L_{j2}+\ldots + L_{j_n}$. Then $L'_j$ is a finitely generated submodule of $L_j$ satisfying the following inclusion
$$\Supp(L_j/L'_j)\subseteq X\setminus\{\p_1,\p_2,\ldots,\p_n\}\subseteq\Max(R)$$
for all $j<t$. For each $j<t$, we set $N_j=H^j_I(M,N/xN)$ and $N'_j=\alpha_j(L'_j)$. Then $N'_j$ is a finitely generated submodule of $N_j$ and the following sequence
\begin{equation}0\sr L_j/L'_j\xrightarrow{\alpha^*_j} N_j/N'_j\xrightarrow{\beta^*_j} (0:x)_{H^{j+1}_I(M,N)}\sr 0\tag{2}\end{equation} 
is exact. We now prove that $L_j$ is minimax for all $j<t$. Look at the exact sequence 
$$\Hom(R/I_M,N_j)\sr \Hom(R/I_M,N_j/N'_j)\sr\Ext^1_R(R/I_M,N'_j).$$ 
For any $j<t$, since $N'_j$ is finitely generated and $\Hom(R/I_M,N_j)$ is finitely generated, so that $\Hom(R/I_M, N_j/N'_j)$ is finitely generated. Hence we obtain by the sequence (2) that $\Hom(R/I_M,L_j/L'_j)$ is finitely generated for all $j<t$. While $\Supp(L_j/L'_j)\subseteq\Max(R)$ and $L_j/L'_j$ is $I_M$-torsion, so that $L_j/L'_j$ is Artinian by \cite[Theorem 1.3]{Me2} for all $j<t$. Thus $L_j$ is minimax for all $j<t$. Consider again the exact sequence (1), that is the following sequence 
\begin{equation}0\sr L_j\xrightarrow{\alpha_j} N_j\xrightarrow{\beta_j} (0:x)_{H^{j+1}_I(M,N)}\sr 0.\tag{1'}\end{equation} 
As $\Hom(R/I_M,N_j)$ is finitely generated for all $j<t$, then so is $\Hom(R/I_M,L_j)$ for all $j<t$. From this, we obtain by \cite[Proposition 4.3]{Me1} that $L_j$ is $I_M$-cofinite for all $j<t$.
Keep in mind that $N_j$ is $I_M$-cofinite for all $j<t-1$. Thus, from the sequence (1'), we have that $(0:x)_{H^j_I(M,N)}$ is $I_M$-cofinite for all $j<t$. In particular, $(0:x)_{H^{t-1}_I(M,N)}$ and $H^{t-1}_I(M,N)/xH^{t-1}_I(M,N)=L_{t-1}$ are $I_M$-cofinite. It implies that $H^{t-1}_I(M,N)$ is $I_M$-cofinite by Lemma \ref{tccofi}. Thus $H^j_I(M,N)$ is $I_M$-cofinite for all $j<t$. On the other hand, by the sequence (1') when $j=t-1$, we have the following exact sequence
$$\Hom(R/I_M,N_{t-1})\sr\Hom(R/I_M,(0:x)_{H^t_I(M,N)})\sr \Ext^1_R(R/I_M,L_{t-1}).$$
Thus, since $\Hom(R/I_M,N_{t-1})$ is finitely generated and $L_{t-1}$ is $I_M$-cofinite, so it yields that 
$\Hom(R/I_M,H^t_I(M,N))=\Hom(R/I_M,(0:x)_{H^t_I(M,N)})$ is finitely generated. Hence the claim is proved, and the proof of  Theorem \ref{Thm1} is complete.
\end{proof}

In \cite[Theorem 2.9]{DS}, Divaani Aazar and Sazeedeh showed that if $\p$ is a prime ideal in a complete local ring $(R,\m)$ with $\dim(R/\p)=1$, then $H^j_\p(M,N)$ is $\p$-cofinite for all $j\ge 0$ whenever $M$ has finite projective dimension. Here, even if without the hypothesis completeness of the ring and the finiteness of projective dimension of $M$, we still obtain the following result.

\begin{corollary}\label{CQ1} If $\dim\Supp(H^j_I(M,N))\le 1$ for all $j$ (this is the case, for example if $\dim(N/I_MN)\le 1$), then $H^j_I(M,N)$ is $I$-cofinite for all $j\ge 0$.
\end{corollary}

We now recall the notion of Bass number: let $K$ be an $R$-module, $i$ an integer and $\p$ a prime ideal, then the $i$-th Bass number $\mu^i(\p,K)$ of $K$ with respect to $\p$ was defined by  $\mu^i(\p,K)=\dim_{k(\p)}(\Ext^i_R(R/\p,K)_\p).$ In  \cite{Kaw2},  S. Kawakami and K. I. Kawasaki proved that if $M$ has finite projective dimension and $\dim(R/I)=1$ then $\mu^i(\p,H^j_I(M,N))$ is finite for all $i, j\ge 0$ and all $\p\in\Spec(R)$. The next corollary is a generalization of this result.

\begin{corollary}\label{sbn} Assume that $\dim\Supp(H^j_I(M,N))\le 1$ for all $j$ (this is the case, for example if $\dim(N/I_MN)\le 1$). Then $\mu^i(\p,H^j_I(M,N))$ is finite for all $i, j\ge 0$ and all $\p\in\Spec(R)$.
\end{corollary}
\begin{proof} If $I\nsubseteq\p$ then $\mu^i(\p,H^j_I(M,N))=0$. If $I\subseteq\p$ then $\Supp(R/\p)\subseteq\Supp(R/I)$, so that $\Ext^i_R(R/\p,H^j_I(M,N))$ is finitely generated for all $i, j$ by Corollary \ref{CQ1} and \cite[Proposition 1]{DM}. Therefore $\mu^i(\p,H^j_I(M,N))$ is finite for all $i, j$, as required.
\end{proof}

\section{Proof of Theorem \ref{Thm3}}

\begin{proof}[\bf Proof of Theorem \ref{Thm3}] 
We first consider the case of $\dim(M)\le 2$. By the short exact sequence 
$0\sr\Gamma_I(M)\sr M\sr\overline M\sr  0$ where $\overline M=M/\Gamma_I(M)$, we get the following exact sequence
$$H^{j-1}_I(\Gamma_I(M),N)\xrightarrow{f_j} H^j_I(\overline M,N)\xrightarrow{g_j} H^j_I(M,N)\xrightarrow{h_j} H^j_I(\Gamma_I(M),N)$$
(following \cite{HZ}). It implies the following exact sequences
$$0\sr\Im f_j\sr H^j_I(\overline M,N)\sr\Im g_j\sr 0\text{ and}$$
$$0\sr\Im g_j\sr H^j_I(M,N)\sr\Im h_j\sr 0.$$
Since $\Gamma_I(M)=(0:I^k)_M$ for some integer $k$, so that 
$$H^j_I(\Gamma_I(M),N)=H^j_{I^k}((0:I^k)_M,N)=\Ext^j_R(\Gamma_I(M),N)$$ 
for all $j$ by Lemma \ref{Lmbs}. Thus $\Im f_j$ and $\Im h_j$ are finitely generated for all $j$. So by the above exact sequences we obtain that $H^j_I(\overline M,N)$ is $I$-cofinite if and only if so is $H^j_I(M,N)$. Therefore we may assume that $\Gamma_I(M)=0$. Then there exists $x\in I$ such that $x$ is an $M$-regular element. From the short exact sequence $0\sr M\xrightarrow x M\sr M/xM\sr 0$ we get the following exact  sequence
$$H^j_I(M/xM,N)\mt (0:x)_{H^j_I(M,N)}\mt 0.$$
Since $\dim (M/xM)\le 1$, so that $\dim\Supp\big((0:x)_{H^j_I(M,N)})\big)\le 1$. Note that $H^j_I(M,N)$ is $I$-torsion and $x\in I$. Thus 
$$\dim\Supp(H^j_I(M,N))=\dim\Supp\big((0:x)_{H^j_I(M,N)})\big)\le 1$$
for all $j$. From this we obtain by Corollary \ref{CQ1} that $H^j_I(M,N)$ is $I$-cofinite for all $j$.

\noindent For the rest of this proof, we consider the case of $\dim(N)\le 2$. By the short exact sequence $0\sr\Gamma_I(N)\sr N\sr\overline N\sr 0$ where $\overline N=N/\Gamma_I(N)$, we get the following exact sequence
$$\Ext^j_R(M,\Gamma_I(N))\xrightarrow{u_j} H^j_I(M,N)\xrightarrow{v_j} H^j_I(M,\overline N)\xrightarrow{w_j} \Ext^{j+1}_R(M,\Gamma_I(N)).$$
It implies the following exact sequences
$$0\sr\Im u_j\sr H^j_I(M,N)\sr\Im v_j\sr 0\text{ and}$$ 
$$0\sr\Im v_j\sr H^j_I(M,\overline N)\sr\Im w_j\sr 0.$$
Thus $\Im u_j$ and $\Im w_j$ are finitely generated for all $j$. So by the above exact sequences we obtain that $H^j_I(M,\overline N)$ is $I$-cofinite if and only if so is $H^j_I(M,N)$. Hence we may assume that $\Gamma_I(N)=0$. So we can take $y\in I$ such that $y$ is an $N$-regular element. From the exact sequence $0\sr N\xrightarrow y N\sr N/yN\sr 0$ we have an exact sequence as follow
$$H^j_I(M,N/yN)\mt (0:y)_{H^{j+1}_I(M,N)}\mt 0$$
for all $j$. So that $\dim\Supp(H^j_I(M,N)) =\dim\Supp\big((0:y)_{H^j_I(M,N)}\big)\le 1$ for all $j\ge 1$. Note that 
$H^0_I(M,N)=\Hom(M,\Gamma_I(N))=\Hom(M,0)=0.$ 
Thus $\dim\Supp(H^j_I(M,N))\le 1$ for all $j$. From this we get by Corollary \ref{CQ1} that $H^j_I(M,N)$ is $I$-cofinite for all $j$, and this finishes the proof of  Theorem \ref{Thm3} .
\end{proof}

As immediate consequences of Theorem \ref{Thm3} we obtain the following results.

\begin{corollary}  If $\dim(R)\le 2$ the $H^j_I(M,N)$ is $I$-cofinite for all $j$, and all finitely generated $R$-modules $M, N$. 
 \end{corollary}
 
\begin{corollary}\label{Dchl}  If $\dim(N)\le 2$ the $H^j_I(N)$ is $I$-cofinite for all $j$. 
\end{corollary}

We next consider furthermore a consequence of Theorem \ref{Thm1} and \ref{Thm3} on the finiteness of associated primes of generalized local cohomology modules. We first recall that an $R$-module $K$ is called {\it weakly Laskerian} if any quotient module of $K$ has finitely many associated primes (cf. \cite{DM}). Note that, all Artinian modules, all finitely generated modules, and all modules with finite support are weakly Laskerian. Moreover, if $0\sr K_1\sr K_2\sr K_3\sr 0$ is an exact sequence, then $K_2$ is weakly Laskerian if and only if $K_1$ and $K_3$ are both weakly Laskerian. Note that if $R$ is a Notherian local ring and $\dim(N)\le 3$ then the third author proved in \cite[Theorem 1.1]{Hoa} that the modules $H^j_I(M,N)$ has only finitely many associated prime ideals for all $j$. In the following, we obtain a stronger result .
 
 \begin{corollary}\label{Pro1} Assume that $(R,\m)$ is a Noetherian local ring. If $\dim(M)\le 3$ or $\dim(N)\le 3$ then $\Ext^i_R(R/I, H^j_I(M,N))$ is weakly Laskerian for all $i, j\ge 0$. In particular, $\Ass_R(H^j_I(M,N))$ is a finite set for all $j\ge 0$.
 \end{corollary}
 \begin{proof} Assume that $\dim(M)\le 3$. By similar arguments as in the proof of Theorem \ref{Thm3}, we obtain the following exact sequences
 $$0\sr\Im f_j\sr H^j_I(\overline M,N)\sr\Im g_j\sr 0\text{ and }$$
$$0\sr\Im g_j\sr H^j_I(M,N)\sr\Im h_j\sr 0,$$
where $\overline M=M/\Gamma_I(M)$. Thus we get the following exact sequences
$$...\sr\Ext^i_R(R/I,\Im f_j)\sr\Ext^i_R(R/I,H^j_I(\overline M,N))\sr\Ext^i_R(R/I,\Im g_j)\sr...\text{ and}$$
$$...\sr\Ext^i_R(R/I,\Im g_j)\sr \Ext^i_R(R/I,H^j_I(M,N))\sr\Ext^i_R(R/I,\Im h_j)\sr...$$
 Moreover, note that $\Im f_j$ and $\Im h_j$ are finitely generated for all $j$. It follows that $\Ext^i_R(R/I,H^j_I(M,N))$ is weakly Laskerian if and only if so is the module $\Ext^i_R(R/I,H^j_I(\overline M,N))$. Therefore we may assume that $\Gamma_I(M)=0$. Thus we get an exact sequence $0\sr M\xrightarrow{x} M\sr M/xM\sr 0$ where $x\in I$ is a regular element of $M$. It implies that $H^j_I(M/xM,N)\sr (0:x)_{H^j_I(M,N)}\sr 0$ is an exact sequence. Hence, as $\dim(M/xM)\le 2$, so we obtain that 
$$\dim\Supp(H^j_I(M,N))\le 2\text{ for all }j\ge 0.$$ 
For the case $\dim(N)\le 3$, by similar arguments as above we may reduce to the hypothesis that $\Gamma_I(N)=0$. Then by the following exact sequence 
$$H^j_I(M,N/yN)\sr (0:y)_{H^{j+1}_I(M,N)}\sr 0$$ with $y\in I$ is an $N-$regular element, 
we get that $\dim\Supp(H^j_I(M,N))\le 2\text{ for all }j\ge 0.$  

\noindent Therefore, for the rest of this proof, we need only to claim the weakly Laskerianness of $\Ext^u_R(R/I,H^v_I(M,N))$ for all $u, v\ge 0$ provided that 
$$\dim\Supp(H^j_I(M,N))\le 2 \text{ for all }j\ge 0.$$
Note that $H^j_I(M,N)\otimes_R\widehat R\cong H^j_{\widehat I}(\widehat M,\widehat N)$. Therefore, in view of \cite[Lemma 2.1]{Mar}, we can assume that $R$ is complete with $\m$-adic topology. We now claim the weakly Laskerianness of $\Ext^u_R(R/I,H^v_I(M,N))$ by way of contradiction. For any integers $u, v$, we set $K=\Ext^u_R(R/I,H^v_I(M,N))$. Assume that there exists a submodule $T$ of $K$ such that $\Ass(K/T)$ is an infinite set. Then there is a countably infinite subset $\{\p_l\}_{l\in\n}$ of $\Ass(K/T)$ such that $\p_l\not=\m$ for all $l\in\n$. Let $S=R\setminus\bigcup_{l\in\n}\p_l$. Then $S$ is a multiplicative closed subset of $R$. Since $\{\p_l\}_{l\in\n}\subseteq\Ass(K/T)$, so that $\{S^{-1}\p_l\}_{l\in\n}\subseteq\Ass_{S^{-1}R}(S^{-1}K/S^{-1}T)$. Thus $\Ass_{S^{-1}R}(S^{-1}K/S^{-1}T)$ is an infinite set.  On the other hand, as $\m\nsubseteq\p_l$ for all $l\in\n$, we get by \cite[Lemma 3.2]{MV} that $\m\nsubseteq\bigcup_{l\in\n}\p_l$; and so that $\m\cap S\not=\emptyset$. It implies that $\dim\Supp(H^j_{S^{-1}I}(S^{-1}M,S^{-1}N))\le 1$ for all $j\ge 0$. From this, we obtain by Corollary \ref{CQ1} that $$S^{-1}K=\Ext^u_{S^{-1}R}(S^{-1}R/S^{-1}I, H^v_{S^{-1}I}(S^{-1}M,S^{-1}N))$$ is finitely generated. It implies that $S^{-1}K/S^{-1}T$ is finitely generated. Hence $\Ass_{S^{-1}R}(S^{-1}K/S^{-1}T)$ is a finite set. On the other hand, by the hypothesis of $T$, the set $\Ass_{S^{-1}R}(S^{-1}K/S^{-1}T)$ is infinite. Hence we obtain a contradiction, and the claim follows. The last conclusion is clear.
\end{proof}

\end{document}